\documentclass[oneside]{amsart}
\usepackage[T1]{fontenc}
\usepackage[latin9]{inputenc}
\usepackage{amsthm}
\usepackage{amssymb}
\usepackage{esint}

\makeatletter
\numberwithin{equation}{section}
\numberwithin{figure}{section}
\theoremstyle{plain}
\newtheorem{thm}{\protect\theoremname}
  \theoremstyle{definition}
  \newtheorem{defn}[thm]{\protect\definitionname}
  \theoremstyle{plain}
  \newtheorem{cor}[thm]{\protect\corollaryname}
  \theoremstyle{plain}
  \newtheorem{lem}[thm]{\protect\lemmaname}

\usepackage[alphabetic]{amsrefs}

\makeatother

  \providecommand{\corollaryname}{Corollary}
  \providecommand{\definitionname}{Definition}
  \providecommand{\lemmaname}{Lemma}
\providecommand{\theoremname}{Theorem}

\begin{document}

\title{On the Lazarev-Lieb Extension of the Hobby-Rice Theorem}

\author{Vermont Rutherfoord}

\address{Vermont Rutherfoord, Department of Mathematical Sciences, Florida
Atlantic University, 777 Glades Road, Boca Raton, FL 33431, USA}

\email{vermont.rutherfoord@gmail.com}
\begin{abstract}
O. Lazarev and E. H. Lieb proved that given $f_{1},\ldots,f_{n}\in L^{1}\left(\left[0,1\right];\mathbb{C}\right)$,
there exists a smooth function $\Phi$ that takes values on the unit
circle and annihilates $\mbox{span}\left\{ f_{1},\ldots,f_{n}\right\} $.
We give an alternative proof of that fact that also shows the $W^{1,1}$
norm of $\Phi$ can be bounded by $5\pi n+1$. Answering a question
raised by Lazarev and Lieb, we show that if $p>1$ then there is no
bound for the $W^{1,p}$ norm of any such multiplier in terms of the
norms of $f_{1},\ldots,f_{n}$.
\end{abstract}
\maketitle
The Hobby-Rice Theorem \cite{hobbyrice} states 
\begin{thm}
\label{thm:hobbyrice}If $f_{1},\ldots,f_{n}\in L^{1}\left(\left[0,1\right];\mathbb{R}\right)$
then there exists $\Phi:\left[0,1\right]\rightarrow\left\{ -1,1\right\} $
with at most $n$ discontinuities such that for each $k$ 
\[
\int_{0}^{1}f_{k}\left(t\right)\Phi\left(t\right)dt=0.
\]

\end{thm}
The theorem has applications in $L^{1}$ approximation and in combinatorics,
particularly necklace splitting problems \cite{nogaalon}. An elegant
proof of the Hobby-Rice Theorem was given by Pinkus \cite{pinkus}
using the Borsuk-Ulam Theorem.

Motivated by a problem in mathematical physics, Lazarev and Lieb \cite{lazarevlieb}
extended this result to obtain a smooth annihilator taking values
on the unit circle, i.e., 
\begin{thm}
\label{thm:smoothhr}If $f_{1},\ldots,f_{n}\in L^{1}\left(\left[0,1\right];\mathbb{C}\right)$
then there exists $\theta\in C^{\infty}\left(\left[0,1\right];\mathbb{R}\right)$
such that 
\begin{equation}
\forall k,\mbox{ }\int_{0}^{1}f_{k}\left(t\right)e^{i\theta\left(t\right)}dt=0.\label{eq:annih}
\end{equation}

\end{thm}
Lazarev and Lieb suggested that there should be a simpler proof, and
in this spirit, we offer the following proof. They also raised the
question of calculating the $H^{1}=W^{1,2}$ norm of $f_{k}e^{i\theta}$.
Corollary~\ref{cor:w11bound} shows there is such $\theta$ with
$\left\Vert e^{i\theta\left(\cdot\right)}\right\Vert _{W^{1,1}}\le5\pi n+1$.
We also show that for $p>1$ there exists a large class of normed
spaces $\mathcal{N}=\left\{ \left(N,\left\Vert \cdot\right\Vert _{N}\right)\right\} $
including $L^{1}$ so that $\left\Vert e^{i\theta\left(\cdot\right)}\right\Vert _{W^{1,p}}$
cannot be bounded by $\left\Vert f_{1}\right\Vert _{N},\ldots,\left\Vert f_{n}\right\Vert _{N}$.
\begin{proof}
[Proof of Theorem 2.]We may assume $f_{1},\ldots,f_{n}$ are linearly
independent in $L^{1}$ and thus choose $0=t_{0}<t_{1}<\ldots<t_{n}<t_{n+1}=1$
so that 
\[
M:=\left[\begin{array}{ccc}
f_{1}\left(t_{1}\right) & \ldots & f_{1}\left(t_{n}\right)\\
\vdots & \ddots & \vdots\\
f_{n}\left(t_{1}\right) & \ldots & f_{n}\left(t_{n}\right)
\end{array}\right]
\]
 is invertible and each $t_{j}$ is a Lebesgue point of all $f_{k}$
(Lemma \ref{lem:gaussian elimination}).

For each $u,v\in\left[-1,1\right]$, let $\theta_{u+iv}:\mathbb{R}\rightarrow\mathbb{R}$
be a step function supported on and increasing in the interval $\left[-1,1\right]$
that takes the values $0,\frac{\pi}{2},\pi,\frac{3\pi}{2}$ on intervals
of lengths $\frac{1+u}{2},\frac{1+v}{2},\frac{1-u}{2},\frac{1-v}{2}$,
respectively. Thus $\int_{-1}^{1}e^{i\theta_{u+iv}\left(t\right)}dt=u+iv$.

Choose $\psi\in C^{\infty}\left(\mathbb{R};\mathbb{R}^{+}\right)$
supported on $\left[-1,1\right]$ such that $\int\psi\left(t\right)dt=1$.
Let $\psi_{h}\left(t\right)=\psi\left(t/h\right)/h$.

Also let $I_{S}$ be the indicator function of $S$. Define $\theta_{h,z}^{\#}=\theta_{z}I_{\left(h-1,1-h\right)}+2\pi I_{\left[1-h,\infty\right)}$
and 
\[
\theta_{h,z}=\begin{cases}
\theta_{z} & \mbox{if }h=0\\
\psi_{h}*\theta_{h,z}^{\#} & \mbox{if }0<h<1
\end{cases}
\]
 Note that if $h>0$ then $\theta_{h,z}\left(-1\right)=0$ and $\theta_{h,z}\left(1\right)=2\pi$,
while $\theta_{h,z}^{\left(m\right)}\left(\pm1\right)=0$ for all
$m\ge1$.

Define $D=\left\{ z\in\mathbb{C}:\left|z\right|\le1\right\} $, $d=\min_{j\in\left\{ 0\ldots n\right\} }\left(t_{j+1}-t_{j}\right)/2$,
and $Q:\left[0,d\right]\times D^{n}\rightarrow\mathbb{C}^{n}$: 
\[
Q\left(h;\vec{z}\right)=\begin{cases}
\left(\begin{array}{c}
\sum_{j=1}^{n}z_{j}\cdot f_{k}\left(t_{j}\right)\end{array}\right)_{k=1\ldots n} & \mbox{if }h=0\\
\,\\
\left(\begin{array}{c}
\sum_{j=1}^{n}\int_{-1}^{1}f_{k}\left(th+t_{j}\right)e^{i\theta_{h,z_{j}}\left(t\right)}dt\end{array}\right)_{k=1\ldots n} & \mbox{if }0<h\le d
\end{cases}
\]
 with $\vec{z}:=\left(z_{1},\ldots,z_{n}\right)$. Since $Q\left(0;\vec{z}\right)=M\left(\vec{z}\right)$,
Lemma~\ref{thm:approx ident} shows there is $\delta\in\left(0,d\right]$
such that for all $\vec{z}\in D^{n}$ 
\[
\vec{z}-M^{-1}\left(Q\left(\delta;\vec{z}\right)\right)\in\frac{1}{2}D^{n}.
\]

Let $L_{_{\delta}}=\left[0,1\right]\backslash\bigcup_{j=1}^{n}\left(t_{j}-\delta,t_{j}+\delta\right)$.
By applying the Hobby-Rice Theorem%
\footnote{The Riemann-Lebesgue Lemma also suffices, but the Hobby-Rice Theorem
enables us to compute a bound of the $W^{1,1}$ norm in Corollary~\ref{cor:w11bound}.%
} to $f_{1}I_{L_{\delta}},\ldots,f_{n}I_{L_{\delta}}$ and smoothing
a finite set of discontinuities, we obtain $\phi\in C^{\infty}\left(\left[0,1\right];\mathbb{R}\right)$
supported on $L_{\delta}$ so that 
\[
\vec{r}:=\left(\int_{L_{\delta}}f_{k}\left(t\right)e^{i\phi\left(t\right)}dt\right)_{k=1\ldots n}\in\frac{\delta}{2}M\left(D^{n}\right).
\]

Since $\phi$ vanishes together with all its derivatives at all $t_{j}\pm\delta$,
for all $\vec{z}\in D^{n}$ 
\[
\theta_{\vec{z}}^{*}\left(t\right)=\begin{cases}
\theta_{\delta,z_{j}}\left(\left(t-t_{j}\right)/\delta\right)+2\pi\left(j-1\right) & \mbox{if }t\in\left[t_{j}-\delta,t_{j}+\delta\right]\mbox{ and }1\le j\le n\\
\phi\left(t\right) & \mbox{if }t\in\left[0,t_{1}-\delta\right)\\
\phi\left(t\right)+2\pi j & \mbox{if }t\in\left(t_{j}+\delta,t_{j+1}-\delta\right)\mbox{ and }1\le j\le n-1\\
\phi\left(t\right)+2\pi n & \mbox{if }t\in\left(t_{n}+\delta,1\right]
\end{cases}
\]
 is in $C^{\infty}\left(\left[0,1\right];\mathbb{R}\right)$. Lemma~\ref{lem:continuity}
establishes the continuity of 
\[
T\left(\vec{z}\right):=\left(\int_{0}^{1}f_{k}\left(t\right)e^{i\theta_{_{\vec{z}}}^{*}\left(t\right)}dt\right)_{k=1\ldots n}=\delta Q\left(\delta;\vec{z}\right)+\vec{r}.
\]

Since $\vec{z}-M^{-1}\left(Q\left(\delta;\vec{z}\right)\right)$ and
$\frac{1}{\delta}M^{-1}\left(\vec{r}\right)$ are in $\frac{1}{2}D^{n}$
for all $\vec{z}\in D^{n}$, then $\vec{z}-\frac{1}{\delta}M^{-1}\left(T\left(\vec{z}\right)\right)\in D^{n}$.
By Brouwer's Fixed Point Theorem, there exists $\vec{z_{_{0}}}\in D^{n}$
such that $\vec{z_{0}}-\frac{1}{\delta}M^{-1}\left(T\left(\vec{z_{0}}\right)\right)=\vec{z_{0}}$,
that is to say $T\left(\vec{z_{0}}\right)=0$.\end{proof}
\begin{defn}
Let 
\begin{eqnarray*}
\left\Vert g\left(\cdot\right)\right\Vert _{W^{1,p}} & = & \left(\int_{0}^{1}\left|g\left(t\right)\right|^{p}dt+\int_{0}^{1}\left|g'\left(t\right)\right|^{p}dt\right)^{\frac{1}{p}}\\
\mbox{and }\left\Vert g\left(\cdot\right)\right\Vert _{\overset{\circ}{W}^{1,p}} & = & \left(\int_{0}^{1}\left|g'\left(t\right)\right|^{p}dt\right)^{\frac{1}{p}}.
\end{eqnarray*}
\end{defn}
\begin{cor}
\label{cor:w11bound}If $f_{1},\ldots,f_{n}\in L^{1}\left(\left[0,1\right];\mathbb{C}\right)$
then there exists $\theta\in C^{\infty}\left(\left[0,1\right];\mathbb{R}\right)$
such that for each $k$ 
\[
\int_{0}^{1}f_{k}\left(t\right)e^{i\theta\left(t\right)}dt=0
\]
 and $\left\Vert e^{i\theta\left(\cdot\right)}\right\Vert _{W^{1,1}}\le5\pi n+1$.\end{cor}
\begin{proof}
The calculation of the bound follows from a careful selection of $\phi$
in the preceding proof. The Hobby-Rice Theorem applied to the $n$
real parts and $n$ imaginary parts of $f_{k}I_{L_{\delta}}$ implies
that there exists $\phi^{\#}:\mathbb{R}\rightarrow\left\{ 0,\pi\right\} $
with at most $2n$ discontinuities such that for each $k$ 
\[
\int f_{k}\left(t\right)I_{L_{\delta}}e^{i\phi^{\#}\left(t\right)}dt=0.
\]
 Since this equation still holds if $\phi^{\#}$ is replaced with
$\pi-\phi^{\#}$, choose such $\phi^{\#}$ that is non-zero on at
most $n$ points at the boundary of $L_{\delta}$. Note that if $\phi^{\#}I_{L_{\delta}}$
is discontinuous at $t$, then either $\phi^{\#}$ is discontinuous
at $t$ or $t=t_{j}\pm\delta$ where $1\le j\le n$ and $\phi^{\#}\left(t\right)\neq0$.
Thus $\phi^{\#}I_{L_{\delta}}$ has at most $3n$ discontinuities.
Choose $\eta>0$ so that by selecting 
\[
\phi=\left(\phi^{\#}I_{L_{\delta+\eta}}\right)*\psi_{\eta}
\]
 then 
\[
\vec{r_{\eta}}:=\left[\int_{L_{\delta}}f_{k}\left(t\right)e^{i\phi\left(t\right)}dt\right]_{k=1\ldots n}\in\frac{\delta}{2}M\left(D^{n}\right).
\]

Note that $\phi$ vanishes together with all its derivatives at all
$t_{j}\pm\delta$. Also note that $\phi^{\#}I_{L_{\delta+\eta}}$
has no more discontinuities than $\phi^{\#}I_{L_{\delta}}$, which
is at most $3n$. Thus there exist $m\le3n$ and $0<y_{1}<\ldots<y_{m}<1$
such that $\phi^{\#}I_{L_{\delta+\eta}}\left(t\right)$ or $\pi-\phi^{\#}I_{L_{\delta+\eta}}\left(t\right)$
for all $t\in\left[0,1\right]\backslash\left\{ y_{1},\ldots,y_{m}\right\} $
is equal to $\pi\sum_{j=1}^{m}\left(-1\right)^{j+1}I_{\left[y_{j},\infty\right)}\left(t\right)$.
Consequently, 
\begin{eqnarray*}
\int_{L_{\delta}}\left|\left(\theta_{\vec{z}}^{*}\right)'\left(t\right)\right|dt & = & \int_{L_{\delta}}\left|\phi'\left(t\right)\right|dt\\
 & = & \int_{0}^{1}\left|\phi'\left(t\right)\right|dt\\
 & = & \int_{0}^{1}\left|\left(\phi^{\#}I_{L_{\delta+\eta}}*\psi_{\eta}\right)'\left(t\right)\right|dt\\
 & \le & \pi\sum_{j=1}^{m}\int_{0}^{1}\left(I_{\left[y_{j},\infty\right)}*\psi_{\eta}\right)'\left(t\right)dt\\
 & \le & 3\pi n
\end{eqnarray*}

Recall that $\theta_{\delta,z}$ is an increasing function with $\theta_{\delta,z}\left(-1\right)=0$
and $\theta_{\delta,z}\left(1\right)=2\pi$ for all $z\in D$. Thus
$\int_{t_{j}-\delta}^{t_{j}+\delta}\left|\left(\theta_{\vec{z}}^{*}\right)'\left(t\right)\right|=2\pi$
for $1\le j\le n$ and so $\int_{0}^{1}\left|\left(\theta_{\vec{z}}^{*}\right)'\left(t\right)\right|dt\le5\pi n$.
Consequently, $\left\Vert e^{i\theta_{\vec{z}}^{*}\left(\cdot\right)}\right\Vert _{\overset{\circ}{W}^{1,1}}\le5\pi n$
and $\left\Vert e^{i\theta_{\vec{z}}^{*}\left(\cdot\right)}\right\Vert _{W^{1,1}}\le5\pi n+1$.
Since $\max_{t\in\left[0,1\right]}\left|\theta_{\vec{z}}^{*}\left(t\right)\right|\le\left(2n+1\right)\pi$,
it follows that $\left\Vert \theta_{\vec{z}}^{*}\left(\cdot\right)\right\Vert _{W^{1,1}}\le\left(7n+1\right)\pi$.
\end{proof}
Clearly, if $f_{1},\ldots,f_{n}$ are real valued, they may be combined
into $\left\lceil \frac{n}{2}\right\rceil $ complex valued functions
and the bounds reduce accordingly.

For $p>1$ the situation is different. 
\begin{defn}
Let 
\[
A\left(f\right)=\left\{ \theta\in C^{\infty}\left(\left[0,1\right];\mathbb{R}\right):\int_{0}^{1}f\left(t\right)e^{i\theta\left(t\right)}dt=0\right\} .
\]

\end{defn}
\begin{defn}
Let 
\[
\rho\left(f\right)=\inf\left\{ \int_{0}^{1}\left|\theta'\left(t\right)\right|^{p}dt:\theta\in A\left(f\right)\right\} .
\]

\end{defn}
\begin{defn}
Let 
\[
\left(\Upsilon_{n}f\right)\left(t\right)=\begin{cases}
f\left(2^{n}t\right) & \mbox{if }0\le t\le2^{-n}\\
0 & \mbox{otherwise}
\end{cases}.
\]
\end{defn}
\begin{thm}
Assume $N$ is a norm for which there exists $f\in L^{1}\left(\left[0,1\right];\mathbb{C}\right)$
such that $0<\left\Vert \Upsilon_{n}f\left(\cdot\right)\right\Vert _{N}<\infty$
for all $n\ge1$ and $\rho\left(f\right)>0$. 

Then, given any $l,K\in\mathbb{R}^{+}$, there exists $g\in L^{1}\left(\left[0,1\right];\mathbb{C}\right)$
such that $\left\Vert g\left(\cdot\right)\right\Vert _{N}=l$ and
$\rho\left(g\right)>K$.\end{thm}
\begin{proof}
Choose $\epsilon>0$ and $\theta\in A\left(\Upsilon_{n}f\right)$
such that $\int_{0}^{1}\left|\theta'\left(t\right)\right|^{p}dt<\rho\left(\Upsilon_{n}f\right)+\epsilon$.
Then $\left(\Upsilon_{-n}\theta\right)|_{\left[0,1\right]}\in A\left(f\right)$,
and so 
\begin{eqnarray*}
\rho\left(\Upsilon_{n}f\right)+\epsilon & > & \int_{0}^{1}\left|\theta'\left(t\right)\right|^{p}dt\\
 & \ge & \int_{0}^{2^{-n}}\left|\theta'\left(t\right)\right|^{p}dt\\
 & = & 2^{-n}\int_{0}^{1}\left|\theta'\left(2^{-n}t\right)\right|^{p}dt\\
 & = & 2^{n\left(p-1\right)}\int_{0}^{1}\left|\left(\theta\left(2^{-n}t\right)\right)'\right|^{p}dt\\
 & \ge & 2^{n\left(p-1\right)}\rho\left(f\right)
\end{eqnarray*}
 proving $\rho\left(\Upsilon_{n}f\right)\ge2^{n\left(p-1\right)}\rho\left(f\right)$.

Also, since $A\left(g\right)=A\left(cg\right)$ for all $c\neq0$
then 
\[
2^{n\left(p-1\right)}\rho\left(f\right)\le\rho\left(\Upsilon_{n}f\right)=\rho\left(l\Upsilon_{n}f/\left\Vert \left(\Upsilon_{n}f\right)\left(\cdot\right)\right\Vert _{N}\right).
\]
 Consequently, if $n$ is large enough so that $2^{n\left(p-1\right)}\rho\left(f\right)>K$
then $g:=l\Upsilon_{n}f/\left\Vert \left(\Upsilon_{n}f\right)\left(\cdot\right)\right\Vert _{N}$
has the property that $\rho\left(g\right)>K$ and $\left\Vert g\left(\cdot\right)\right\Vert _{N}=l$.
\end{proof}
The $W^{1,p}$ norms of $f_{k}\left(t\right)e^{i\theta\left(t\right)}$
fare no better, since if $f_{1}\left(t\right)=1$, then $\left\Vert f_{1}\left(\cdot\right)e^{i\theta\left(\cdot\right)}\right\Vert _{W^{1,p}}=\left\Vert e^{i\theta\left(\cdot\right)}\right\Vert _{W^{1,p}}\ge\left(\max_{k}\rho\left(f_{k}\right)\right)^{\frac{1}{p}}$.

\section*{Lemmas}

We include the lemmas that were used above, some or all of which may
be familiar to the reader.
\begin{lem}
\label{lem:gaussian elimination}If $f_{1},\ldots,f_{n}\in L^{1}\left(\left[0,1\right];\mathbb{C}\right)$
are linearly independent in $L^{1}$, then there exist $t_{1},\ldots t_{n}\in\left(0,1\right)$
so that 
\[
M:=\left[\begin{array}{ccc}
f_{1}\left(t_{1}\right) & \ldots & f_{1}\left(t_{n}\right)\\
\vdots & \ddots & \vdots\\
f_{n}\left(t_{1}\right) & \ldots & f_{n}\left(t_{n}\right)
\end{array}\right]
\]
 is invertible and $t_{j}$ is a Lebesgue point of $f_{k}$ for each
$j,k\in1\ldots n$.\end{lem}
\begin{proof}
Let $P$ be the set of points in $\left(0,1\right)$ that are Lebesgue
points for all $f_{k}$.

The case $n=1$ is clear. If $n>1$, let us assume inductively that
there are $t_{1},\ldots,t_{n-1}\in P$ such that $M':=\left[f_{k}\left(t_{j}\right)\right]_{\left(n-1\right)\times\left(n-1\right)}$
is invertible. Thus there exist $\beta_{1},\ldots,\beta_{n-1}\in\mathbb{C}$
such that 
\[
\left[\beta_{1}\ldots\beta_{n-1}\right]M'=\left[f_{n}\left(t_{1}\right)\ldots f_{n}\left(t_{n-1}\right)\right].
\]
 Furthermore, since $f_{1},\ldots,f_{n}$ are linearly independent
in $L^{1}$, there exists $t_{n}\in P$ such that 
\[
y_{n}:=f_{n}\left(t_{n}\right)-\beta_{1}f_{1}\left(t_{n}\right)-\ldots-\beta_{n-1}f_{n-1}\left(t_{n}\right)\neq0.
\]
 Thus $M:=\left[f_{k}\left(t_{j}\right)\right]_{n\times n}$ has a
nonzero determinant, namely $y_{n}\det M'$.\end{proof}
\begin{lem}
\label{thm:approx ident}If $t_{0}\in\left(0,1\right)$ is a Lebesgue
point of $f\in L^{1}\left(\left[0,1\right];\mathbb{C}\right)$, then,
uniformly in $z\in D$, 
\[
\lim_{h\rightarrow0^{+}}\int_{-1}^{1}f\left(th+t_{0}\right)\cdot\theta_{h,z}\left(t\right)dt=f\left(x\right)\cdot z
\]
\end{lem}
\begin{proof}
Given $\epsilon>0$, let $\delta_{1}<\epsilon/20\pi\left(\left|f\left(t_{0}\right)\right|+1\right)$.
If $0<h<\delta_{1}$, then, since $\theta_{z}\left(t\right)$ is a
step function, for all values of $t\in\left(2h-1,1-2h\right)$ that
are not within distance $h$ of a discontinuity of $\theta_{z}$,
$\theta_{h,z}\left(t\right)=\theta_{z}\left(t\right)$. Since there
are at most three discontinuities of $\theta_{z}$ in $\left(2h-1,1-2h\right)$
and $\theta_{h,z}\left(t\right),\theta_{z}\left(t\right)\in\left[0,2\pi\right]$,
\[
2\pi\cdot6h\ge\int_{2h-1}^{1-2h}\left|\theta_{h,z}\left(t\right)-\theta_{z}\left(t\right)\right|dt
\]
 and so 
\begin{eqnarray*}
\int_{-1}^{1}\left|e^{i\theta_{h,z}\left(t\right)}-e^{i\theta_{z}\left(t\right)}\right|dt & \le & \int_{-1}^{1}\left|\theta_{h,z}\left(t\right)-\theta_{z}\left(t\right)\right|dt\\
 & \le & 2\pi\cdot\left(6+4\right)h\\
 & < & \epsilon/\left(\left|f\left(t_{0}\right)\right|+1\right).
\end{eqnarray*}

Choose $\delta_{2}>0$ so that if $0<h<\delta_{2}$ then 
\[
\int_{-1}^{1}\left|f\left(th+t_{0}\right)-f\left(t_{0}\right)\right|dt<\epsilon/2
\]
 and let $\delta=\min\left\{ \delta_{1},\delta_{2}\right\} $. Then
\begin{eqnarray*}
 &  & \left|\int_{-1}^{1}f\left(th+t_{0}\right)\cdot e^{i\theta_{h,z}\left(t\right)}dt-\overbrace{\int_{-1}^{1}f\left(t_{0}\right)\cdot e^{i\theta_{z}\left(t\right)}dt}^{=f\left(t_{0}\right)\cdot z}\right|\\
 & \le & \left|\int_{-1}^{1}\left(f\left(th+t_{0}\right)-f\left(t_{0}\right)\right)e^{i\theta_{h,z}\left(t\right)}dt\right|+\left|\int_{-1}^{1}f\left(t_{0}\right)\left(e^{i\theta_{h,z}\left(t\right)}-e^{i\theta_{z}\left(t\right)}\right)dt\right|\\
 & \le & \int_{-1}^{1}\left|f\left(th+t_{0}\right)-f\left(t_{0}\right)\right|dt+\left|f\left(t_{0}\right)\right|\int_{-1}^{1}\left|e^{i\theta_{h,z}\left(t\right)}-e^{i\theta_{z}\left(t\right)}\right|dt\\
 & < & \epsilon
\end{eqnarray*}
\end{proof}
\begin{lem}
\label{lem:continuity}If $f\in L^{1}\left(\left[0,1\right];\mathbb{C}\right)$,
$t_{0}\in\left(0,1\right)$, and $0<h<\min\left\{ t_{0},1-t_{0}\right\} $,
then \textup{
\[
q\left(z\right):=\int_{-1}^{1}f\left(th+t_{0}\right)e^{i\theta_{h,z}\left(t\right)}dt
\]
 is a continuous function of $z\in D$.}\end{lem}
\begin{proof}
Since $\theta_{h,z}^{\#}$ is a step function whose intervals of constancy
vary continuously with $z$, for $\epsilon>0$, there exists $\delta>0$
such that if $\left|z_{1}-z_{2}\right|<\delta$ then 
\[
\int_{-1}^{1}\left|\theta_{h,z_{1}}^{\#}\left(t\right)-\theta_{h,z_{2}}^{\#}\left(t\right)\right|dt<\epsilon/\left(\left\Vert f\right\Vert _{L^{1}}+1\right)h\left\Vert \psi\left(t\right)\right\Vert _{L^{\infty}}.
\]
 Then 
\begin{eqnarray*}
\left\Vert e^{\theta_{h,z_{1}}\left(t\right)}-e^{\theta_{h,z_{2}}\left(t\right)}\right\Vert _{L^{\infty}} & \le & \left\Vert \theta_{h,z_{1}}\left(t\right)-\theta_{h,z_{2}}\left(t\right)\right\Vert _{L^{\infty}}\\
 & < & \epsilon/\left(\left\Vert f\right\Vert _{L^{1}}+1\right)
\end{eqnarray*}
 and so 
\begin{eqnarray*}
\left|q\left(z_{1}\right)-q\left(z_{2}\right)\right| & \le & \int_{-1}^{1}\left|f\left(th+t_{0}\right)\left(e^{i\theta_{h,z_{1}}\left(t\right)}-e^{i\theta_{h,z_{2}}\left(t\right)}\right)\right|dt\\
 & < & \epsilon.
\end{eqnarray*}

\end{proof}

\end{document}